\documentclass[12pt, a4paper]{amsart}
\usepackage{amsfonts}
\usepackage{amscd,amsmath}
\usepackage{amssymb}
\usepackage{setspace}
\newenvironment{rcases}
  {\left.\begin{aligned}}
  {\end{aligned}\right\rbrace}

\newtheorem{theorem}{Theorem}[section]

\newtheorem{proposition}[theorem]{Proposition}
\newtheorem{definition}[theorem]{Definition}

\theoremstyle{remark}

\newtheorem*{rem*}{Remark}

\newcommand{\C}{\mathbb{C}}

\newcommand{\T}{\mathbb{S}}

\newcommand{\comment}[1]{}
\numberwithin{equation}{section}
\begin{document}

\title{A construction of bent functions on a finite group}
\author{Mani Shankar Pandey$^1$, Sumit Kumar Upadhyay$^2$ AND Vipul Kakkar$^3$\vspace{.4cm}
}
\address{$^{1,2}$Department of Applied Sciences,\\ Indian Institute of Information Technology, Allahabad }
\address{$^{3}$Department of Mathematics\\Central University of Rajasthan\\
Kishangarh, Rajasthan}
\thanks {2010 Mathematics Subject classification : 05B10, 05C25, 20C15}
\email{\tiny{$^1$manishankarpandey4@gmail.com, $^2$upadhyaysumit365@gmail.com, $^3$vplkakkar@gmail.com}}.\\

\begin{abstract}
In this paper, we discuss when a class function on a finite group is a bent function. We have found a necessary condition for a class function on a finite abelian group to be bent. Also, we have found a necessary and sufficient condition for a class function on a finite cyclic group to be bent. 
\end{abstract}
\maketitle
\textbf{Keywords}:  Bent functions, finite group, class functions, characters.

\section{Introduction}
Bent functions was introduced by O. Rothuaus in \cite{or}. Bent functions are maximum distant from the affine functions. Bent functions have many applications in various fields, viz. cryptography and coding theory and combinatorial design. The maximum-length sequences based on bent functions have applications in a code division multiple access (CDMA) environment. T. Wada in \cite{tw}, studied the relation between bent functions and codes for CDMA. The concept of bent functions on a finite abelian group was introduced by , O. A. Logachev et. al in \cite{ol}. V. I. Solodovnikov in \cite{vs}, generalized the concept of bent functions from an abelian group to another abelian group. L. Poinsot in \cite{LP}, introduced the concept of a bent function on a finite nonabelian group as a generalization of those of a finite abelian group.
The following are basic definition and result from \cite{LP}.
\begin{definition}
 Let $G$ be a finite group and $f : G \longrightarrow \mathbb{C}$. The derivative of $f$ in the direction $\sigma \in G$ is defined as $\frac{df}{d \sigma} : G \longrightarrow \mathbb{C}$, $x \longrightarrow \overline{f(x)}f(\sigma x)$.
\end{definition}
\begin{theorem}\label{t}
Let $G$ be a finite group and $f : G\rightarrow$ $\T^1$. The map $f$ is bent if and only if $\displaystyle{\sum_{x \in G} \frac{df}{d \sigma}(x)}$=0, $\forall \sigma \in G^\ast = G \setminus \{e\}$.
\end{theorem}

\section{Main Results}

Let $G$ be a finite group of order $n$. A class function is a function on a group $G$ that is constant on the conjugacy classes of $G$. Let $\chi_1,\chi_2,....\chi_r$, where $1\leq r \leq n$ be the all irreducible distinct characters of $G$. Since $\chi_1,\chi_2,....\chi_r$ form an orthonormal basis of the set of all class functions on $G$, every class functions from $G \rightarrow$ $\T^1$ can be written as linear combination of $\chi_1,\chi_2,....\chi_r$. In this paper we are interested when a class function from $G \rightarrow$ $\T^1$ is a bent function, where $\T^1=\{z \in \C\mid |z|=1\}$.
\begin{theorem}
Let $G$ be a finite abelian group of order $n$ and $\chi_1,\chi_2,....\chi_n$ be its distinct one dimensional characters. If a class function $f :G \rightarrow$ $\T^1$ written as $f=a_1 \chi_1+a_2 \chi_2+...+a_n \chi_n$ is bent, then $\vert a_i \vert ^2 = \frac{1}{n}$, $\forall i=1,2....n$, where $a_i$ are complex numbers.
\end{theorem}
\begin{proof}
Since $f$ is bent, by Theorem \ref{t}
\begin{center}
$\displaystyle{\sum_{x \in G} \frac{df}{d \sigma}(x)}$=0 $\forall \sigma \in G^\ast~ =~ G \setminus \{e\}$.
\end{center}
 Now
$\displaystyle{\sum_{x \in G} \frac{df}{d \sigma}(x)}=  \displaystyle{\sum_{x \in G} \overline{f(x)}f(\sigma x)} \\= \displaystyle{\sum_{x \in G}} (\overline {a_1\chi_1(x)+a_2 \chi_2(x)+.....a_n \chi_n(x)})((a_1\chi_1(\sigma x)+a_2\chi_2(\sigma x)+.....a_n\chi_n(\sigma x))$ = $\displaystyle{\sum_{x \in G}}  [\vert a_1 \vert^2 \overline{\chi_1(x)} \chi_1(x\sigma)+ \cdots+\vert a_n \vert^2 \overline{\chi_n(x)} \chi_n(x\sigma)]$+ $ \displaystyle{\sum_{x \in G}} [\overline{a_1}a_2 \overline\chi_1(x) \chi_2(\sigma x)+\cdots +\overline{a_1}a_n \overline\chi_1(x) \chi_n(\sigma x)]+\cdots +\sum_{x \in G} [\overline{a_n}a_1 \overline{\chi_n(x)}\chi_1(\sigma x)+\cdots+ \overline {a_n} a_{n-1} \overline\chi_n(x) \chi_{n-1}(\sigma x)]=0$ $\forall \sigma \in G^*$.

Since $\chi_i$ are a group homomorphism, by using orthogonality relations in above equation we have
\begin{equation}\label{e1}
\sum_{i=1}^n \vert a_i \vert^2 \chi_i(\sigma)=0, ~ \forall \sigma \in G^\ast
\end{equation}
Also since $f$ is defined from $G$ to $\T^1$, $\vert f(x) \vert^2=1$ $\forall x \in G$. Therefore $\displaystyle{\sum_{x \in G}} f(x) \overline{f(x)}= \vert G \vert$. This implies that
\begin{center} 

$\displaystyle{\sum_{x \in G}} (\overline{a_1\chi_1(x)+a_2\chi_2(x)+...a_n\chi_n(x)}) (a_1 \chi_1(x)+a_2 \chi_2(x)+...a_n\chi_n(x))= \vert G \vert$
\end{center}
Therefore again by using orthogonality relations we have
\begin{center}
$\sum_{i=1}^n \vert a_i \vert ^2=1$
\end{center}
Since $\chi_i(e)=1$ $\forall i=1,\cdots, n$, we can write above equation as
\begin{equation}\label{e2}
\sum_{i=1}^n\vert a_i \vert^2 \chi_i(e)=1, \text{where $e$ is the identity of $G$}.
\end{equation}  

Let $G =\{e, g_2, \cdots, g_n\}$. Then we have a non-homogeneous system of linear equations $\mathbf{\Phi} \mathbf{w}=\mathbf{y}$, where
\begin{equation}\label{e3}
\underbrace{\begin{bmatrix}
\chi_1(e)       &   \chi_2(e)       & \dots     &   \chi_n(e)       \\
\chi_1(g_2)       &   \chi_2(g_2)       & \dots     &   \chi_n(g_2)       \\
\vdots  &  \vdots   &   \vdots  &   \vdots  \\   
\chi_1(g_n)       &   \chi_2(g_n)       & \dots     &   \chi_n(g_n)       \\
            \end{bmatrix}
            }_{\mathbf{\Phi}}
\underbrace{\begin{bmatrix}
\vert a_1 \vert^2     \\
\vert a_2 \vert^2     \\
\vdots  \\
\vert a_n \vert^2   \\
            \end{bmatrix}
            }_{\mathbf{w}}
    =
\underbrace{\begin{bmatrix}
1     \\
0     \\
\vdots  \\
0     \\
            \end{bmatrix}
            }_{\mathbf{y}}
\end{equation}
Since columns of the matrix $\mathbf{\Phi}$ are orthogonal to each other(by orthogonality relations), the matrix $\mathbf{\Phi}$ is non-singular with inverse $\frac{1}{n}{\bar{\mathbf{\Phi}}}^T$, where $T$ denotes transpose. This implies that the above system of equations has a unique solution.

Therefore 
\begin{center}
$\mathbf{w} = \mathbf{\Phi}^{-1} \mathbf{y} $  \[\Rightarrow\underbrace{\begin{bmatrix}
\vert a_1 \vert^2    \\
\vert a_2 \vert^2    \\
\vdots  \\
\vert a_n \vert^2     \\
            \end{bmatrix}
            }_{\mathbf{w}} = \frac{1}{n}
{\bar{\mathbf{\Phi}}}^T
\underbrace{\begin{bmatrix}
1     \\
0     \\
\vdots  \\
0   \\
            \end{bmatrix}
            }_{\mathbf{y}}
\]
\end{center}
Thus, we have $\vert a_i \vert^2=\frac{1}{n}$ $\forall i=1,\cdots,n$.
\end{proof}

It gives only the necessary condition for a class function to be bent, this condition is not sufficient.
As if we consider the cyclic group $\mathbb{Z}_2 =\lbrace \overline{0},\overline{1}\rbrace$ of order $2$ then its characters $\chi_1$ and $\chi_2$ are defied as 
\begin{center}
$\chi_1(\overline{0})=1$ and $\chi_1(\overline{1})=1$

$\chi_2(\overline{0})=1$ and $\chi_2(\overline{1})=-1$
\end{center}
Therefore it can be easily seen that if we take $a_i$=$\frac{1}{•\sqrt{2}}$ $\forall i=1,2$ and f=$a_1\chi_1+a_2\chi_2$ in this case our function is not defined from $G$ to $\T^1$, bent is so far away.

Let $G = \mathbb{Z}_n$ be a cyclic group of order $n$ and $\omega$ be the $n$-th roots of unity. Then $G$ has $n$ irreducible characters $\chi_1,\chi_2,....\chi_n$ defined by $\chi_i (g) = \omega^{i-1}$, for all $1\leq i \leq n$, where $\omega = e^{2\pi \mathrm{i}/n}$. Below, we obtain a necessary and sufficient condition for a class function on a finite cyclic group to be bent function.
\begin{proposition}\label{p1}
A class function $f= a_1 \chi_1+a_2 \chi_2 +a_3 \chi_3$ on $\mathbb{Z}_3$ is bent if and only if $\vert a_1 \vert=\vert a_2 \vert =\vert a_3 \vert = \frac{1}{\sqrt{3}}$ and $\overline{a_1}a_2+\overline{a_2}a_3+\overline{a_3}a_1=0$, where $\chi_1, \chi_2$ and $\chi_3$ are irreducible characters of $\mathbb{Z}_3$.
\end{proposition}
\begin{proof} The character table of $\mathbb{Z}_3$ is 
\begin{center}
\begin{tabular}{|p{1cm}|p{1cm}|p{1cm}|p{1cm}|}
 & $\bar{0}$ & $\bar{1}$ & $\bar{2}$\\\hline
$\chi_1$ & $1$ & $1$ & $1$\\
\hline
$\chi_2$ & $1$ & $\omega$ & $\omega^2$\\
\hline
$\chi_3$ & $1$ & $\omega^2$ & $\omega$\\
\hline
\end{tabular}
\end{center}

Therefore, we have $f (\bar{0}) = a_1 + a_2 + a_3$, $f (\bar{1}) = a_1 + a_2 \omega + a_3 \omega^2$ and $f (\bar{2}) = a_1 + a_2 \omega^2 + a_3 \omega$.

\vspace{0.2 cm}

Note that $f(\bar{0}) \in \T^1 \Leftrightarrow \overline{f(\bar{0})}f(\bar{0})=1$ 

\quad \quad $ \Leftrightarrow \overline{(a_1+a_2 +a_3)}(a_1+a_2 +a_3)= 1 $ 

\quad \quad $\Leftrightarrow \vert a_1 \vert^2 + \vert a_2 \vert^2 + \vert a_3 \vert^2 +\overline{a_1}a_2+\overline{a_2}a_3+\overline{a_3}a_1 + a_1\overline{a_2} + a_2\overline{a_3} + a_3\overline{a_1} = 1$ 

\vspace{0.2 cm}

Similarly, $f(\bar{1}) \in \T^1 \Leftrightarrow \overline{f(\bar{1})}f(\bar{1})=1$

\quad $\Leftrightarrow \vert a_1 \vert^2 + \vert a_2 \vert^2 + \vert a_3 \vert^2 +\omega (\overline{a_1}a_2+\overline{a_2}a_3+\overline{a_3}a_1 ) + \omega^2 (a_1\overline{a_2} + a_2\overline{a_3} + a_3\overline{a_1}) = 1$

\vspace{0.2 cm}

and $f(\bar{2}) \in \T^1 $ 

 $\Leftrightarrow \vert a_1 \vert^2 + \vert a_2 \vert^2 + \vert a_3 \vert^2 +\omega (a_1\overline{a_2} + a_2\overline{a_3} + a_3\overline{a_1}) + \omega^2 (\overline{a_1}a_2+\overline{a_2}a_3+\overline{a_3}a_1 )  = 1$ 

\vspace{0.2 cm}

Now, suppose that $f$ is bent. Then by Theorem 2.1, we have $\vert a_1 \vert = \vert a_2 \vert = \vert a_3 \vert =\frac{1}{\sqrt{3}}$. By above conditions, we have the following

\begin{equation}\label{p1e1} 
\overline{a_1}a_2+\overline{a_2}a_3+\overline{a_3}a_1 + a_1\overline{a_2} + a_2\overline{a_3} + a_3\overline{a_1} = 0
\end{equation}
\begin{equation}\label{p1e2} 
\omega (\overline{a_1}a_2+\overline{a_2}a_3+\overline{a_3}a_1 ) + \omega^2 (a_1\overline{a_2} + a_2\overline{a_3} + a_3\overline{a_1}) = 0
\end{equation}

\begin{equation}\label{p1e3} 
\omega (a_1\overline{a_2} + a_2\overline{a_3} + a_3\overline{a_1}) + \omega^2 (\overline{a_1}a_2+\overline{a_2}a_3+\overline{a_3}a_1 ) = 0
\end{equation}

On solving Equations \ref{p1e1}, \ref{p1e2} and \ref{p1e3}, we have $\overline{a_1}a_2+\overline{a_2}a_3+\overline{a_3}a_1 = a_1\overline{a_2} + a_2\overline{a_3} + a_3\overline{a_1} =0$

\vspace{0.2 cm}

Conversely, suppose $\vert a_1 \vert=\vert a_2 \vert =\vert a_2 \vert = \frac{1}{\sqrt{3}}$ and $\overline{a_1}a_2+\overline{a_2}a_3+\overline{a_3}a_1=0$. Thus by above conditions, it is clear that $f(x) \in \T^1$ for all $x \in \mathbb{Z}_3$.

Using orthogonality relations, we have
  
$$\sum_{x \in \mathbb{Z}_3}\frac{df}{d \sigma}(x) = \vert a_1 \vert^2\chi_1(\sigma) + \vert a_2 \vert^2 \chi_2(\sigma) + + \vert a_3 \vert^2 \chi_3(\sigma) = \frac{1}{3} (1 + \omega + \omega^2) = 0$$ for $\sigma \in \mathbb{Z}_3 \setminus \{\bar{0}\}$. Thus $f$ is bent.
\end{proof}
\begin{proposition}\label{p2}
A class function $f= a_1 \chi_1+a_2 \chi_2 +a_3 \chi_3 +a_4 \chi_4$ on $\mathbb{Z}_4$ is bent if and only if $\vert a_1 \vert=\vert a_2 \vert =\vert a_3 \vert =\vert a_4 \vert = \frac{1}{\sqrt{4}}$ and $\overline{a_1}a_2+\overline{a_2}a_3+\overline{a_3}a_4 +\overline{a_4}a_1= 0 ~\&~ \overline{a_1}a_3+\overline{a_3}a_1 + \overline{a_2}a_4 +\overline{a_4}a_2 = 0$, where $\chi_1, \chi_2, \chi_3$ and $\chi_4$ are irreducible characters of $\mathbb{Z}_4$.
\end{proposition}
\begin{proof} The character table of $\mathbb{Z}_4$ is
\begin{center}
 \begin{tabular}{|p{1cm}|p{1cm}|p{1cm}|p{1cm}|p{1cm}|}
 & $\bar{0}$ & $\bar{1}$ & $\bar{2}$ & $\bar{3}$\\\hline
$\chi_1$ & $1$ & $1$ & $1$ & $1$\\
\hline
$\chi_2$ & $1$ & $\mathrm{i}$ & $-1$ & $-\mathrm{i}$\\
\hline
$\chi_3$ & $1$ & $-1$ & $1$ & $-1$\\
\hline
$\chi_4$ & $1$ & $-\mathrm{i}$ & $-1$ & $\mathrm{i}$\\
\hline
\end{tabular} 
\end{center}

Therefore, we have $f (\bar{0}) = a_1 + a_2 + a_3 + a_4$, $f (\bar{1}) = (a_1 - a_3)+ \mathrm{i} (a_2 -a_4)$ and $f (\bar{2}) = (a_1 - a_2) + (a_3 - a_4)$ and $f (\bar{3}) = (a_1 - a_3)- \mathrm{i} (a_2 -a_4)$.

\vspace{0.2 cm}

Note that $f(\bar{0}) \in \T^1  \Leftrightarrow \vert a_1 \vert^2 + \vert a_2 \vert^2 + \vert a_3 \vert^2 + \vert a_4 \vert^2 +\overline{a_3}a_1 + a_3\overline{a_1}+\overline{a_4}a_2 + a_4\overline{a_2}+ \overline{a_1}a_2+\overline{a_2}a_3+ a_1\overline{a_2} + a_2\overline{a_3} + \overline{a_1}a_4+\overline{a_3}a_4+ a_1\overline{a_4} + a_3\overline{a_4} = 1$,

\vspace{0.2 cm}

$f(\bar{1}) \in \T^1 \Leftrightarrow \vert a_1 \vert^2 + \vert a_2 \vert^2 + \vert a_3 \vert^2 + \vert a_4 \vert^2 -\overline{a_3}a_1 - a_3\overline{a_1}- \overline{a_4}a_2 - a_4\overline{a_2}+ \mathrm{i} (\overline{a_1}a_2+\overline{a_2}a_3+\overline{a_3}a_4+ a_1\overline{a_4} - a_1\overline{a_2} -a_2\overline{a_3} - \overline{a_1}a_4-  a_3\overline{a_4}) = 1$, 

\vspace{0.2 cm}

$f(\bar{2}) \in \T^1 \Leftrightarrow \vert a_1 \vert^2 + \vert a_2 \vert^2 + \vert a_3 \vert^2 + \vert a_4 \vert^2 +\overline{a_3}a_1 + a_3\overline{a_1}+\overline{a_4}a_2 + a_4\overline{a_2}- \overline{a_1}a_2-\overline{a_2}a_3- a_1\overline{a_2}- a_2\overline{a_3} - \overline{a_1}a_4-\overline{a_3}a_4- a_1\overline{a_4} - a_3\overline{a_4} = 1$,

\vspace{0.2 cm}

and $f(\bar{3}) \in \T^1  \Leftrightarrow \vert a_1 \vert^2 + \vert a_2 \vert^2 + \vert a_3 \vert^2 + \vert a_4 \vert^2 -\overline{a_3}a_1 - a_3\overline{a_1}- \overline{a_4}a_2 - a_4\overline{a_2}- \mathrm{i} (\overline{a_1}a_2+\overline{a_2}a_3+\overline{a_3}a_4+ a_1\overline{a_4} - a_1\overline{a_2} -a_2\overline{a_3} - \overline{a_1}a_4-  a_3\overline{a_4}) = 1$.

Now, suppose that $f$ is bent. Then by Theorem 2.1, we have $\vert a_1 \vert = \vert a_2 \vert = \vert a_3 \vert = \vert a_4 \vert =\frac{1}{\sqrt{4}}$. By above conditions, we have the following.

\begin{equation}\label{p2e1}
 \overline{a_3}a_1 + a_3\overline{a_1}+\overline{a_4}a_2 + a_4\overline{a_2}+ \overline{a_1}a_2+\overline{a_2}a_3+ a_1\overline{a_2} + a_2\overline{a_3} + \overline{a_1}a_4+\overline{a_3}a_4+ a_1\overline{a_4} + a_3\overline{a_4} = 0
\end{equation}
 \begin{equation}\label{p2e2}
-\overline{a_3}a_1 - a_3\overline{a_1}- \overline{a_4}a_2 - a_4\overline{a_2}+ \mathrm{i} (\overline{a_1}a_2+\overline{a_2}a_3+\overline{a_3}a_4+ a_1\overline{a_4} - a_1\overline{a_2} -a_2\overline{a_3} - \overline{a_1}a_4-  a_3\overline{a_4}) = 0 
\end{equation}
\begin{equation}\label{p2e3}
\overline{a_3}a_1 + a_3\overline{a_1}+\overline{a_4}a_2 + a_4\overline{a_2}- \overline{a_1}a_2-\overline{a_2}a_3- a_1\overline{a_2}- a_2\overline{a_3} - \overline{a_1}a_4-\overline{a_3}a_4- a_1\overline{a_4} - a_3\overline{a_4} = 0 
\end{equation}
\begin{equation}\label{p2e4}
-\overline{a_3}a_1 - a_3\overline{a_1}- \overline{a_4}a_2 - a_4\overline{a_2}- \mathrm{i} (\overline{a_1}a_2+\overline{a_2}a_3+\overline{a_3}a_4+ a_1\overline{a_4} - a_1\overline{a_2} -a_2\overline{a_3} - \overline{a_1}a_4-  a_3\overline{a_4}) = 0.
\end{equation}
Adding Equations \ref{p2e1} and \ref{p2e3}, we get
\begin{equation}\label{p2e5}
 \overline{a_1}a_3+\overline{a_3}a_1 + \overline{a_2}a_4 +\overline{a_4}a_2 = 0
\end{equation}

By Equations \ref{p2e2} and \ref{p2e5} or \ref{p2e4} and \ref{p2e5}, we have 
\begin{equation}\label{p2e6}
\Rightarrow \overline{a_1}a_2+\overline{a_2}a_3+\overline{a_3}a_4+ a_1\overline{a_4} = a_1\overline{a_2} +a_2\overline{a_3} + \overline{a_1}a_4+  a_3\overline{a_4}.
\end{equation}
 By Equations \ref{p2e1}, \ref{p2e5} and \ref{p2e6},  we have 
$$\overline{a_1}a_2+\overline{a_2}a_3+\overline{a_3}a_4 +\overline{a_4}a_1= 0.$$

Conversely, suppose $\vert a_1 \vert=\vert a_2 \vert =\vert a_3 \vert =\vert a_4 \vert = \frac{1}{\sqrt{4}}$ and $\overline{a_1}a_2+\overline{a_2}a_3+\overline{a_3}a_4 +\overline{a_4}a_1= 0 ~\&~ \overline{a_1}a_3+\overline{a_3}a_1 + \overline{a_2}a_4 +\overline{a_4}a_2 = 0$. Thus by above conditions,  $f(x) \in \T^1$ for all $x \in \mathbb{Z}_4$.

Using orthogonality relations, we have\\  
$$\sum_{x \in \mathbb{Z}_4}\frac{df}{d \sigma}(x) = \vert a_1 \vert^2\chi_1(\sigma) + \vert a_2 \vert^2 \chi_2(\sigma) + \vert a_3 \vert^2 \chi_3(\sigma) + \vert a_4 \vert^2 \chi_4(\sigma)$$ = $\frac{1}{4} (\chi_1(\sigma) + \chi_2(\sigma)+ \chi_3(\sigma) + \chi_4(\sigma)) = 0$ for $\sigma \in \mathbb{Z}_4 \setminus \{\bar{0}\}$. Thus $f$ is bent.
\end{proof}

The character table of cyclic group $\mathbb{Z}_n$ of order $n$ is \\
\begin{center}
\begin{tabular}{|p{1cm}|p{1cm}|p{1cm}|p{1cm}|p{2cm}|}
 & $\bar{0}$ & $\bar{1}$ & $\cdots$ & $\bar{n}$\\\hline
$\chi_1$ & $1$ & $1$ & $\cdots$ & $1$\\
\hline
$\chi_2$ & $1$ & $\omega$ & $\cdots$ & $\omega^{n-1}$\\
\hline
$\chi_3$ & $1$ & $\omega^2$ & $\cdots$ & $\omega^{2(n-1)}$\\
\hline
$\vdots$ & $\vdots$ & $\vdots$ & $\vdots$ & $\vdots$\\
\hline
$\chi_n$ & $1$ & $\omega^{n-1}$ & $\cdots$ & $\omega^{(n-1)(n-1)}$\\
\hline
\end{tabular} where $\omega = e^{\frac{2\pi \mathrm{i}}{n}}$
\end{center}

We now state a necessary and sufficient condition for a class function on a finite cyclic group to be bent. The proof of the theorem is on same line  as those in Propositions \ref{p1} and \ref{p2}.

\begin{theorem}
A class function $f= a_1 \chi_1+a_2 \chi_2 +\cdots+a_n\chi_n$ on $\mathbb{Z}_n$ is bent if and only if $\vert a_1 \vert=\vert a_2 \vert =\cdots=\vert a_n\vert = \frac{1}{\sqrt{n}}$ and 
\[
\begin{rcases}
\overline{a_1}a_2+\overline{a_2}a_3+\cdots+\overline{a}_{n-2}a_{n-1}+\overline{a}_{n-1} a_n+\overline{a_n} a_1= 0 \\
 \overline{a_1}a_3+\overline{a_2}a_4+\cdots+\overline{a}_{n-2}a_n+\overline{a}_{n-1}a_1+\overline{a_n}a_2= 0\\ \vdots\\
 \overline{a_1}a_{(\frac{n-1}{2})+1}+\overline{a_2}a_{(\frac{n-1}{2})+2}+\cdots+\overline{a}_{(\frac{n-1}{2})+1}a_n+\overline{a}_{(\frac{n-1}{2})+2}a_1+ \cdots + \overline{a_n}a_{(\frac{n-1}{2})}= 0
\end{rcases}
\text{if $n$ is odd}
\] and
\[
\begin{rcases}
\overline{a_1}a_2+\overline{a_2}a_3+\cdots+\overline{a}_{n-2}a_{n-1}+\overline{a}_{n-1} a_n+\overline{a_n} a_1= 0 \\
 \overline{a_1}a_3+\overline{a_2}a_4+\cdots+\overline{a}_{n-2}a_n+\overline{a}_{n-1}a_1+\overline{a_n}a_2= 0\\ \vdots\\
 \overline{a_1}a_{(\frac{n}{2})+1}+\overline{a_2}a_{(\frac{n}{2})+2}+\cdots+\overline{a}_{(\frac{n}{2})+1}a_n+\overline{a}_{(\frac{n}{2})+2}a_1+ \cdots + \overline{a_n}a_{(\frac{n}{2})}= 0
\end{rcases}
\text{if $n$ is even.}
\] 
\end{theorem}




\begin{rem*}
One can similarly prove that a class function $f= a_1 \chi_1+a_2 \chi_2 +a_3 \chi_3 +a_4 \chi_4$ on Klein's four group $V_4$ is bent if and only if $\vert a_1 \vert=\vert a_2 \vert =\vert a_3 \vert =\vert a_4 \vert = \frac{1}{\sqrt{4}}$ and $\overline{a_1}a_2 + \overline{a_3}a_4+\overline{a_2}a_1 + \overline{a_4}a_3= 0,  \overline{a_1}a_3+ \overline{a_2}a_4+\overline{a_3}a_1 + \overline{a_4}a_2 = 0 ~\&~ \overline{a_1}a_3+\overline{a_3}a_1 + \overline{a_2}a_4 +\overline{a_4}a_2 = 0$, where $\chi_1, \chi_2, \chi_3$ and $\chi_4$ are irreducible characters of $V_4$.
\end{rem*}

Now, we show there is no class function on the symmetric group $S_3$ on three symbols which is bent.
\begin{proposition}
There is no class function $f : S_3 \rightarrow \T^1$ which is bent.
\end{proposition}
\begin{proof}
The character table of $S_3$ is 
\begin{center}
\begin{tabular}{|p{1cm}|p{1cm}|p{1cm}|p{1cm}|}
 & $I$ & $(12)$ & $(123)$\\\hline
$\chi_1$ & $1$ & $1$ & $1$\\
\hline
$\chi_2$ & $1$ & $-1$ & $1$\\
\hline
$\chi_3$ & $2$ & $0$ & $-1$\\
\hline
\end{tabular}
\end{center}

Let $f$ be a class function on $S_3$ and $f=a_1\chi_1+a_2\chi_2+a_3\chi_3$, where $\chi_i$ are irreducible characters of $S_3$ and $a_i$ are complex numbers for all $i=1,2,3$. Suppose $f$ is bent. Then by Theorem \ref{t}
\begin{equation}\label{5}
\sum_{x \in S_3}\frac{df}{d\sigma}(x) =0 
\end{equation}
for all $\sigma \in G^*$.

Therefore, by Equation \ref{5}, we have

$\sum_{x \in G} \sum_{i=1}^3 \vert a_i \vert ^2 \overline{\chi_i(x)} \chi_i(\sigma x) $\\ $ +  \overline{ a_1\chi_(x)} (a_2\chi_2(\sigma x)+a_3\chi_3(\sigma x)) + \overline{a_2\chi_2(x})(a_1\chi_1(\sigma x)+a_3\chi_(\sigma x)) + \overline{a_3\chi_3(x)}(a_1\chi(\sigma x)+a_2\chi_2(\sigma x))=0$.
\vspace{0.2 cm}

Since characters are class function therefore $\chi_i((12))=\chi_i((23))=\chi_i((13))$
and $\chi_i((123))=\chi_i((132))$ $ \forall$ $i=1,2,3$.
Since from character table of $S_3$ we know that $\chi_1(x)=1$ $ \forall$  $x\in G$

$\chi_2(I)=1$, $\chi_2(12)=-1$ and $\chi_2(123)=1,$

$,\chi_3(I)=2$, $\chi_3((12))=0$ and $\chi_3((123))=-1$
Now if we solve Equation \ref{5} for $\sigma =(12)$ we have
\begin{equation}\label{7}
\begin{aligned}
(\overline{a_1 + a_2 + 2a_3})(a_1 - a_2) + (\overline{a_1 - a_2})(a_1+a_2+2 a_3) + 2(\overline{a_1 - a_2})(a_1 + a_2 - a_3)     
         +\\ 2(\overline{a_1 +a_2 - a_3})(a_1 - a_2)=0
         \end{aligned}
\end{equation}

Similar equations will be obtained if we solve Equation \ref{5} for $\sigma =(23)$ and $\sigma =(13)$ 

On simplifying above Equation \ref{7}
\begin{equation}\label{*}
\vert a_1 \vert^2 = \vert a_2 \vert^2
\end{equation}        

similarly if we solve \ref{5} for $ \sigma =(123)$ or (132), we have
\begin{equation}\label{8}
\begin{aligned}
(\overline{a_1+a_2+2a_3})(a_1+a_2-a_3)+(\overline{a_1+a_2-a_3})(a_1+a_2+2a_3)+3(\overline{a_1-a_2})(a_1-a_2)\\+(\overline{a_1+a_2-a_3})(a_1+a_2-a_3)=0
\end{aligned}
\end{equation} 

On simplifying above Equation \ref{8}, we have

\begin{equation}\label{**}
2(\vert a_1 \vert^2+\vert a_2 \vert^2)=\vert a_3\vert^2
\end{equation} 

Similar equation will be obtained if we solve \ref{5} for $\sigma =(132)$.

Now, since $ \vert f(x) \vert^2=1$ $\forall x \in G $
$\Rightarrow \overline{f(x)}f(x)=1$ $\forall x \in G $

$$\sum_{x \in G} (\overline{a_1\chi_1(x)+a_2\chi(x)+a_3\chi(x)})(a_1\chi_1(x)+a_2\chi_2(x)+a_3\chi_3(x))$$ =  order of $G$

using orthogonality relations we have

\begin{equation}\label{***}
\vert a_1 \vert^2+\vert a_2 \vert^2 + \vert a_3 \vert^2=1 
\end{equation}
Now on solving Equations \ref{8}, \ref{**} and \ref{***}, we have

\begin{equation}\label{9}
\vert a_1 \vert^2=\vert a_2 \vert^2 = \frac{1}{6}, \vert a_3 \vert^2=\frac{2}{3}
\end{equation}
Since $\vert f((1 2))\vert^2 = 1 \Rightarrow (\overline{a_1-a_2})(a_1-a_2) = 1 \Rightarrow \vert a_1 \vert^2 + \vert a_2 \vert^2 - (\overline{a_1}a_2 + \overline{a_2}a_1) = 1 \Rightarrow \overline{a_1}a_2 + \overline{a_2}a_1 = \frac{-2}{3}$ 

Now, we claim that that there is no $a_1, a_2 \in \mathbb{C}$ with $\vert a_1 \vert^2=\vert a_2 \vert^2 = \frac{1}{6}, \overline{a_1}a_2 + \overline{a_2}a_1 = \frac{-2}{3}$. 

Suppose we have $a_1, a_2 \in \mathbb{C}$ with $\vert a_1 \vert^2=\vert a_2 \vert^2 = \frac{1}{6}, \overline{a_1}a_2 + \overline{a_2}a_1 = \frac{-2}{3}$. Hence $\langle (a_1, a_2), (a_2, a_1) \rangle = \frac{-2}{3} \Rightarrow |\langle (a_1, a_2), (a_2, a_1) \rangle| = \frac{2}{3}$.

By Cauchy-Schwartz Inequality, $|\langle (a_1, a_2), (a_2, a_1) \rangle| \leq  ||(a_1, a_2)||||(a_2, a_1)|| = \sqrt{\vert a_1 \vert^2 + \vert a_2 \vert^2}\sqrt{\vert a_1 \vert^2 + \vert a_2 \vert^2}$. Therefore we have $\frac{2}{3} \leq \frac{1}{3}$, which is not possible. Thus the claim is proved. 

Hence there is no class function $f : S_3 \rightarrow \T^1$ which is bent.
\end{proof}
\begin{proposition}
Necessary condition for a class function to be bent on the quaternion group $Q_8$ is 
$\vert a_i \vert^2 =\vert a_j \vert^2 = \frac{2}{9}, \forall 1\leq i,j\leq 4$ and $\vert a_5 \vert^2=\frac{1}{9}$
\end{proposition}
\begin{proof}The character table of the group $Q_8$ is \\
\begin{center}
\begin{tabular}{|p{1cm}|p{1cm}|p{1cm}|p{1cm}|p{1cm}|p{1cm}|}
 & $1$ & $-1$ & $i$ & $j$ & $k$\\\hline
$\chi_1$ & $1$ & $1$ & $1$ & $1$ & $1$\\
\hline
$\chi_2$ & $1$ & $1$ & $1$ & $-1$ & $-1$\\
\hline
$\chi_3$ & $1$ & $1$ & $-1$ & $-1$ & $1$\\
\hline
$\chi_4$ & $1$ & $1$ & $-1$ & $1$ & $-1$\\
\hline
$\chi_5$ & $2$ & $-2$ & $0$ & $0$ & $0$\\
\hline
\end{tabular}

\end{center}

Let $f : Q_8\rightarrow \T^1$ be a class function such that $f$ is bent. Then derivative of $f$ in the direction of $\sigma$ is balanced for all $\sigma \in Q_8^* $ i.e.
\begin{center}
$$\sum_{x \in Q_8}\frac{df}{d \sigma}(x)=0.....\forall \sigma \in Q_8^*$$
\end{center}
Since $f$ is a class function, we have $f=a_1\chi_1+a_2\chi_2+a_3\chi_3+a_4\chi_4+a_5\chi_5$ where $a_i$ are complex numbers and $\chi_i$ are irreducible characters of $Q_8$ for all $1 \leq i \leq 5$.
We have $$ \sum_{x \in Q_8} \sum_{i=1}^4 \vert a_i \vert^2\chi_i(\sigma)+\sum_{x \in Q_8}\vert a_5 \vert^2 \overline{\chi_5(x)}\chi_5(\sigma x)+$$

$$\overline{a_1}a_5(\chi_5(\sigma)+\chi_5(-\sigma)+\chi_5(i\sigma)+\chi_5(j\sigma)+\chi_5(k\sigma)+\chi_5(-i\sigma)+\chi_5(-j\sigma)+\chi_5(-k\sigma))$$

$$+\overline{a_2}a_5(\chi_5(\sigma)+\chi_5(-\sigma)+\chi_5(i\sigma)-\chi_5(j\sigma)-\chi_5(k\sigma)+\chi_5(-i\sigma)-\chi_5(-j\sigma)-\chi_5(-k\sigma))$$

$$+\overline{a_3}a_5(\chi_5(\sigma)+\chi_5(-\sigma)-\chi_5(i\sigma)-\chi_5(j\sigma)+\chi_5(k\sigma)-\chi_5(-i\sigma)-\chi_5(-j\sigma)+\chi_5(-k\sigma))$$

$$+\overline{a_4}a_5(\chi_5(\sigma)+\chi_5(-\sigma)-\chi_5(i\sigma)+\chi_5(j\sigma)-\chi_5(k\sigma)-\chi_5(-i\sigma)+\chi_5(-j\sigma)-\chi_5(-k\sigma))$$=0

Now solving above equation for $\sigma=-1,i,j$ and $k$.

We have the following equations.
\begin{equation}\label{18}
\vert a_1 \vert^2+\vert a_2 \vert^2+\vert a_3 \vert^2+\vert a_4 \vert^2-8\vert a_5 \vert^2=0, \text{ when} \sigma =-1
\end{equation}

similarly if we put $\sigma = i,-i,j,-j,k,-k$ we have
\begin{equation}\label{19}
\vert a_1 \vert^2+\vert a_2 \vert^2-\vert a_3 \vert^2-\vert a_4\vert^2=0
\end{equation}
\begin{equation}\label{20}
\vert a_1 \vert^2-\vert a_2 \vert^2-\vert a_3 \vert^2+\vert a_4\vert^2=0
\end{equation}
\begin{equation}\label{21}
\vert a_1 \vert^2-\vert a_2 \vert^2+\vert a_3 \vert^2-\vert a_4\vert^2=0
\end{equation}
Also we know that $\vert f(x) \vert^2=1$, for all $x \in Q_8$. 
Solving this using orthogonality relations we have
\begin{equation}\label{22}
\sum_{i=1}^5 \vert a_i \vert^2=1
\end{equation}

Now on solving Equations \ref{18} - \ref{22}, we have

$\vert a_i \vert^2 =\vert a_j \vert^2 = \frac{2}{9}, \forall 1\leq i,j\leq 4$ and 

$\vert a_5 \vert^2=\frac{1}{9}$

\end{proof}
\textbf{Acknowledgement}: We would like to thank Professor Ramji Lal, HRI, Allahabad for his valuable suggestions, discussions and constant support. The first author thanks Indian Institute of Information Technology, Allahabad for providing institute fellowship.


\begin{thebibliography}{00}
\bibitem{ol} O. A. Logachev, A. A. Sal’nikov, and V. V. Yashchenko, Bent Functions
on a Finite Abelian Group, Dis. Math. Appl., 7 (6), 547-564, 1997.
\bibitem{LP}
L. Poinsot, Bent functions on a finite nonabelian group, J. Dis. Math. Sci. and Crypto., vol. 9 (2), pp. 349-364, 2006.
\bibitem{or} O. Rothaus, On Bent Functions, J.Combin. Theory Ser. A, 20 (3), 300-305, 1976.
\bibitem{vs} V. I. Solodovnikov, Bent Functions from a Finite Abelian Group to a Finite
Abelian Group, Dis. Math. Appl., 12 (2), 111-126 (2002).
\bibitem{tw} T. Wada, Characteristic of Bit Sequences Applicable to Constant Amplitude Orthogonal Multicode systems, IEICE Trans. Fundamentals E83-A (11), 2160-2164 (2000).
\end{thebibliography}
\end{document}